\documentclass[10pt]{article}

\usepackage{preamble}

\begin{document}

\newcommand\restr[2]{{
  \left.\kern-\nulldelimiterspace 
  #1 
  \vphantom{\big|} 
  \right|_{#2} 
  }}

\makeatletter
\renewcommand{\@seccntformat}[1]{%
  \ifcsname prefix@#1\endcsname
    \csname prefix@#1\endcsname
  \else
    \csname the#1\endcsname\quad
  \fi}
\makeatother

\makeatletter
\newcommand{\colim@}[2]{%
  \vtop{\m@th\ialign{##\cr
    \hfil$#1\operator@font colim$\hfil\cr
    \noalign{\nointerlineskip\kern1.5\ex@}#2\cr
    \noalign{\nointerlineskip\kern-\ex@}\cr}}%
}
\newcommand{\colim}{%
  \mathop{\mathpalette\colim@{\rightarrowfill@\textstyle}}\nmlimits@
}
\makeatother

\newcommand\rightthreearrow{%
        \mathrel{\vcenter{\mathsurround0pt
                \ialign{##\crcr
                        \noalign{\nointerlineskip}$\rightarrow$\crcr
                        \noalign{\nointerlineskip}$\rightarrow$\crcr
                        \noalign{\nointerlineskip}$\rightarrow$\crcr
                }%
        }}%
}

\title{The spectrum of strict units of topological modular forms}         
\author{Zhenpeng Li and Kiran Luecke}        
\date{\today}          
\maketitle

\begin{abstract}
In the theory of spectral algebraic geometry, few objects receive as much study as the spectrum of topological modular forms. In this paper, we compute the strict units of topological modular forms, defined as the connective cover of the mapping spectrum from $\Z$ to the units spectrum.
\end{abstract}

\tableofcontents

\section{Introduction}
\subsection{Background and Main Results}

In classical algebra, the functor sending a commutative ring $R$ to its group of units $R^\times$ is corepresented by the free commutative ring on one invertible generator:
\[R^\times\simeq \mathrm{Hom}_\mathrm{CRing}(\Z[x^\pm],R).\]
Moreover, this corepresenting object coincides with the group algebra of the infinite cyclic group $\mathbb{Z}$. In homotopical algebra, the functor sending a commutative ring spectrum $R$ to its spectrum of units $\glone R$ is also corepresented by the free commutative ring spectrum on one invertible\footnote{Of degree 0.} generator
\[\glone R\simeq \Map_{\calg} (\S\{x^\pm\}, R)\]
but the corepresenting object is quite far from being the spherical group algebra $\mathbb{S}[\Z]$. Instead, one has a second notion of units, called the spectrum of \emph{strict} units $\G_m(R)$, which satisfies the equation
\[\G_m(R):= \Map_{\calg} (\S[\Z], R)\simeq\map_\Sp(\mathbb{Z},\glone R).\]
In many cases, it is the strict units which have favorable properties. For starters, $\S[\Z]$ is flat over $\S$, something which fails dramatically for $\S\{x^\pm\}$. Next, the functor $\G_m$ detects the chromatic height of $R$, and its torsion cousins $\mu_{p^k}(-):=\Map_{\calg} (\S[\Z/p^k],-)$, the ``higher roots of unity'', play a central role in recent duality results such as the chromatic Fourier transform \cite{BCSY}. Moreover, the strict units of spherical group rings are known to be well behaved (cf. \cite{CarmeliPic}, \cite{CNY}).
From a different angle, the units $\glone R$ control the twists for $R$-cohomology, but it is the strict ones which are usually of geometric interest: for example, the most naturally occurring twists of topological $K$-theory $\KU$ are those which come from degree 3 integral cohomology classes\footnote{Often via line-bundle-valued Cech cocycles}, which are manifestly strict units.

In this paper, we will consider the case $R=\TMF$, the spectrum of \emph{topological modular forms}. This was constructed by Goerss, Hopkins, and Miller (cf. \cite{HopkinsTMF}), by lifting the structure sheaf of the \'etale site of $\mathcal{M}_\mathrm{ell}$, the moduli stack of elliptic curves, to a sheaf of $\E_\infty$ ring spectra. The $\mathbb{E}_\infty$ ring $\TMF$ is then defined as the global sections of this sheaf. Thus, by construction, $\TMF$ is intimately related to the algebraic geometry of elliptic curves, and there is a forgetful map $\pi_*\TMF\rightarrow \MF_*$ to the classical ring of modular forms (graded by twice the weight). A comprehensive collection of results regarding $\TMF$ is \cite{TMFbook}. Somewhat mysterious, $\TMF$ is also related to the mathematical physics of 2-dimensional quantum field theories, and this relation has been an area of intense research since its conception by Segal, Stolz, and Teichner (cf. \cite{Segal}, \cite{StolzTeichner}). In short, performing spectral algebro-geometric constructions with $\TMF$ is a central tool for current researchers who range from homotopy theorists to mathematical physicists to geometric representation theorists.

In the present paper, we give a complete calculation of $\G_m\TMF$, with the hope of aiding these researchers in their programs. Our main result is the following:

\begin{introthm}\label[thm]{ThmA}
The following table describes the homotopy type of $\G_m\TMF$; the numbers on the right indicate the degree of the corresponding Eilenberg--Maclane spectrum, and $k> 1$.
\begin{align*}
    \G_m\TMF =
    &(\Z/2)^\infty &[0] \\
    &\prod_p\Z_p^{N_p} &[2] \\
    &\prod_p\Z_p^{N_p}\oplus A &[3] \\
    & A&[4k-1].
\end{align*}`
in which $A$ denotes the rational vector space defined as the cofiber of the obvious map
\[\Q\otimes \Z[x]\rightarrow \Q\otimes\prod_p\Z[x]_p,\]
and $N_p$ denotes the number of $\mathrm{Gal}(\Fpb/\F_p)$-orbits of isomorphism classes of supersingular elliptic curves over $\Fpb$. See \cref{mainthm} in \cref{sec7.2} for more details.
\end{introthm}

There are two variants of $\TMF$ which are of interest in the literature. The first is $\Tmf$, which is defined just like $\TMF$ with $\mathcal{M}_\mathrm{ell}$ replaced by its 1-point compactification. The second is $\tmf$, the connective cover of $\Tmf$. Thus there are canonical maps
\[\tmf\rightarrow\Tmf\rightarrow\TMF\]
and they induce isomorphisms
\[\tmf[(\Delta^{24})^{-1}]\simeq\Tmf[(\Delta^{24})^{-1}]\simeq\TMF,\]
where $\Delta^{24}$ is a certain homotopy class of degree 576 lifting the 24th power of the modular discriminant $\Delta\in \MF_*$.
For the sake of completeness, we record the following.
\begin{introthm}\label[thm]{ThmB}
    The canonical map
    \[\G_m\tmf\to\G_m\TMF\]
    identifies the source as the connected cover of the target.
\end{introthm}

\subsection{Proof Strategy and Organization}
To achieve these results, we first apply some standard techniques in stable homotopy theory, the arithmetic and chromatic fracture squares, to reduce the calculation to certain chromatic completions of $\TMF$. In this completed setting, there exist robust calculational tools involving power operations and descent results, which convert the computation of $\G_m$ to some classical facts about elliptic curves and modular forms.

The outline of the paper is as follows. \cref{sec2} provides a
necessary background for the classical theory on supersingular elliptic curves and modular forms. Then we treat the $K(2)$-local case in \cref{sec3} using descent results. Next, in \cref{sec4} we recall some facts about $K(1)$-local $\mathbb{E}_\infty$-ring spectra and their power operations. These tools will be employed in \cref{sec5} and \cref{sec6} to study the strict units of $L_{K(1)}\TMF$  and $L_{K(1)}L_{K(2)}\TMF$, respectively.
Finally, by assembling the relevant fracture squares, we complete the calculation of $\G_m\TMF$ in \cref{sec7}.

The calculation presented in this paper results from a synthesis of power operations and descent theory, each of which plays an indispensable role in the $K(1)$- and $K(2)$-local computations, respectively. The authors anticipate that the study of power operations for $K(n)$-local $\mathbb{E}_\infty$-rings at higher chromatic heights may serve as a new method for the computation of $\G_mL_{K(2)}\TMF$ and the global counterpart.

Despite the abstract equivalence between the strict units spectra of $\Tmf$ and $\TMF$ at many chromatic localizations, the corresponding canonical map fails to be an equivalence upon passage to $L_{K(1)}$-completion at prime 2. In \cref{rem_tmf_TMF_K(1)}, we show that this failure reflects the distinction of algebro-geometric objects defining these two $\mathbb{E}_\infty$-ring spectra. Moreover, an interesting asymptotic phenomenon occurs---the integer $N_p$ appearing in higher homotopy groups of $\G_m\TMF$ tends to infinity as $p$ grows, which is by no means obvious from the definition. We hope that the study of homotopical algebra might in turn shed new light on the enumerative problems of elliptic curves.

\subsection{Acknowledgements}
We thank David Gepner for expressing interest in the outcome of this calculation and encouraging us to do it. We thank Tilman Bauer, Rudradip Biswas, Shachar Carmeli, Langwen Hui, Hana Jia Kong, Guchuan Li, Lennart Meier, Charles Rezk, Ningchuan Zhang and Peter Rongping Zhou for helpful conversations and suggestions. We thank Hana Jia Kong for comments on a previous draft. We thank Alexander Lajeunesse for catching a group cohomology error in a previous draft.

\section{Facts about elliptic curves and modular forms}\label{sec2}
This section is intended to collect all the facts needed on elliptic curves in our paper, in particular the fundamental properties of supersingular elliptic curves over finite fields. Let us begin with the definition.
\begin{prop}\label[prop]{prop:supersingular_elliptic_curves}
    For an elliptic curve $E$ over a field $k$ of characteristic $p$, the following are equivalent.
    \begin{enumerate}
        \item $E[p]$ is connected.
        \item $[p]:E\to E$ is purely inseparable, and the $j$-invariant $j(E)\in \mathbb{F}_{p^2}$.
        \item The formal group associated to $E$ is of height 2.
        \item\label{supersingular} $\#E(\mathbb{F}_p)=p+1$, if additionally $k=\mathbb{F}_p$ and $p\geq5$.
    \end{enumerate}
   The elliptic curve $E$ satisfying any of these conditions is called supersingular. Otherwise, $E$ is called ordinary.
\end{prop}
\begin{proof}
    See \cite[Theorem {\uppercase\expandafter{\romannumeral5}}.3.1]{silverman2009arithmetic}.
\end{proof}

\begin{example}\label[example]{example:supersingular_elliptic_curves_at_p=2}
When $k=\overline{\mathbb{F}}_2$, there is exactly one isomorphism class of supersingular elliptic curve. For instance, one can pick the equation $y^2+y=x^3$, and its $j$-invariant is equal to 0.
\end{example}
\begin{example}\label[example]{example:supersingular_elliptic_curves_at_p=3}
    When $k=\overline{\mathbb{F}}_3$, there is still only one isomorphism class of supersingular elliptic curve (e.g., $y^2=x^3-x$) with $j$-invariant 0.
\end{example}

Over $\overline{\mathbb{F}}_p$, the Frobenius automorphism induces an action on the isomorphism classes of elliptic curves, which preserves the property of being supersingular or not. To be more precise, it will send a supersingular elliptic curve $E$ to the relative Frobenius $E^{(p)}$. Hence, supersingular elliptic curves with $j$-invariants in $\mathbb{F}_p$ will be fixed, while others will be packed into an orbit of order 2 since $j(E^{(p^2)})=(j(E))^{p^2}=j(E)$.

In order for the computation in the sequel, we also need to count the number of supersingular elliptic curves and the number of orbits.
\begin{notation}\label[notation]{notation:Np}
    For a prime number $p$, denote by $S_p$ the collection of isomorphism classes of supersingular elliptic curves over $\overline{\mathbb{F}}_p$, and by $N_p$ the number of orbits from the Frobenius action.
\end{notation}
\begin{rem}
    The numbers $\#S_p$ and $N_p$ have already been studied to some extent when $p>3$: Using the class number $h(d)$ of the imaginary quadratic field $\mathbb{Q}(\sqrt{d})$, we have
    \begin{equation*}
\#S_p=\left[\frac{p}{12}\right]+\left\{
\begin{aligned}
    &0\enspace\mathrm{if}\enspace p\equiv 1\,(\mathrm{mod}\,12)\\
    &2\enspace\mathrm{if}\enspace p\equiv 11\,(\mathrm{mod}\,12)\\
    &1\enspace\mathrm{otherwise}.
\end{aligned}
\right.
\end{equation*}

\begin{equation*}
N_p=\frac{1}{2}\#S_p+
    \begin{cases}
        \frac{1}{4}h(-4p) &\mathrm{if}\enspace p\equiv 1\,(\mathrm{mod}\,4)\\
        h(-p)&\mathrm{if}\enspace p\equiv3\,(\mathrm{mod}\,8)\\
        \frac{1}{2}h(-p)&\mathrm{if}\enspace p\equiv 7\,(\mathrm{mod}\,8).\\
    \end{cases}
\end{equation*}
A standard proof was recorded in \cite[Theorem \uppercase\expandafter{\romannumeral5}.4.1]{silverman2009arithmetic} and \cite[Theorem 14.18]{cox2022primes}\footnote{For $N_p$, we should use \cref{prop:supersingular_elliptic_curves}.(\ref{supersingular}) to apply the formula therein.}. Besides, Ogg \cite{Ogg1980} also showed that the number of orbits of order 2 is equal to the genus of the modular curve $X^+_0(p)$, and that only for the prime divisor of the Monster group is the genus equal to 0.
\end{rem}

The following table lists the numbers $\#S_p$ and $N_p$ for small primes $p$, from which we can see that 37 is the first prime with $N_p<\#S_p$. And both of them approach $\infty$ as $p$ grows.

\begin{table}[H]
    \centering
    \begin{tabular}{|c|c|c|c|c|c|c|c|c|c|c|c|c|}
    \hline
     $p$&2&3&5&7&11&13&17&19&23&29&31&37\\
     \hline
     $\#S_p$&1&1&1&1&2&1&2&2&3&3&3&3\\ 
     \hline
     $N_p$&1&1&1&1&2&1&2&2&3&3&3&2\\
     \hline
    \end{tabular}
    \caption*{Numbers of $\#S_p$ and $N_p$, where 37 is the first prime such that $N_p<\#S_p$.}
\end{table}

Additionally, the automorphism groups of elliptic curves are easily classified in full generality, as the following proposition shows.

\begin{prop}\label[prop]{prop:automorphism_of_elliptic_curves}
Let $k$ be a field containing $\mathbb{F}_{p^2}$. For any elliptic curve $E$ over $k$, the automorphism group $\mathrm{Aut}(E)$ is classified as follows:
\begin{equation*}
    \mathrm{Aut}(E)=
    \begin{cases}
        \Z/2&\mathrm{if}\enspace j\neq0,1728\\
        \Z/6&\mathrm{if}\enspace j=0\enspace\mathrm{and}\enspace p\neq2,3\\
        \Z/4&\mathrm{if}\enspace j=1728\enspace\mathrm{and}\enspace p\neq2,3\\

        \Z/3\ltimes Q_8&\mathrm{if}\enspace j=0\enspace\mathrm{and}\enspace p=2\\
        
        \Z/4\ltimes \Z/3&\mathrm{if}\enspace j=0\enspace\mathrm{and}\enspace p=3.
    \end{cases}
\end{equation*}
\end{prop}
\begin{proof}
    See \cite[Section \uppercase\expandafter{\romannumeral5}.10]{silverman2009arithmetic}.
\end{proof}

In the sequel, we need to analyze the value of the cyclotomic character at specific elements in the Morava stabilizer group, which is in turn related to the Dieudonn\'{e} theory of $p$-divisible groups. Regarding a supersingular elliptic curve $E$ over a finite field $k$, recall that the Dieudonn\'{e} module $\mathrm{DM}(E[p^\infty])$ is a free $\mathbb{W}(k)$-module of rank 2 (see \cite[Section 1.3]{lurie2013ambidexterity}). Hence, every endomorphism of $E$ admits a determinant associated to the action on $\mathrm{DM}(E[p^\infty])$.

\begin{prop}\label[prop]{prop:determinant_of_automorphisms}
    For any endomorphism $\varphi$ of a supersingular elliptic curve $E$ over a finite field $k$, the corresponding determinant on the Dieudonn\'{e} module $\mathrm{DM}(E[p^\infty])$ is equal to the degree of $\varphi$, $\mathrm{deg}(\varphi)$.
\end{prop}
\begin{proof}
    According to \cite[Example 1.4.18]{lurie2013ambidexterity}, for a $p$-divisible group, the Dieudonn\'{e} module of its Cartier dual is the $\mathbb{W}(k)$-linear dual of the original Dieudonn\'{e} module. In other words, we have
    \[\mathrm{DM}(E[p^\infty]^\vee)\simeq\mathrm{Hom}_{\mathbb{W}(k)}(\mathrm{DM}(E[p^\infty]),\mathbb{W}(k)).\]
    As a consequence, the Weil pairing $E[p^\infty]\simeq E[p^\infty]^\vee$, as a pairing defined on the $p$-divisible group $E[p^\infty]$, endows the Dieudonn\'{e} module with a nondegenerate skew form $e(-,-)$. For the bilinear form $e(-,-)$ on $\mathrm{DM}(E[p^\infty])$, we know ${\mathrm{deg}(\varphi)}e(v_1,v_2)=e(\varphi(v_1),\varphi(v_2))$, since the similar formula holds in terms of Weil pairing. Then the proposition follows from simple algebraic deduction:

    Assume that $v_1,v_2$ form a basis of $\mathrm{DM}(E[p^\infty])$, and that $\varphi$ acts on this basis by the matrix
    $\begin{pmatrix}
        a&b\\
        c&d
    \end{pmatrix}.$ Then we have the following chain of equations
    \[\mathrm{deg}(\varphi)e(v_1,v_2)=e(av_1+cv_2,bv_1+dv_2)=(ad-bc)e(v_1,v_2),\]
    by bilinearity and skewness.
\end{proof}

Last but not least, we will conclude this section with some identities from modular form theory that will also appear in $\pi_*\TMF$.
\begin{prop}\label[prop]{prop_modular_form_identities}
The following identities hold:
    \begin{enumerate}[label={$(\arabic*)$}]
        \item The $j$-invariant $j=\frac{c_4^3}{\Delta}$, where $j$, $c_4$ and $\Delta$ are regarded as elements in the graded ring $MF_*\simeq\Z[c_4,c_6,\Delta^\pm]/(c_4^3-c_6^2-1728\Delta)$.
        \item At infinity, the $j$-invariant admits the $q$-expansion as follows
        \[j=q^{-1}+744+196884q+21493760q^2+\dots.\]
        \item The normalized Eisenstein series $E_{p-1}$ is of the form $E_4^{\epsilon_1}E_6^{\epsilon_2}\Delta^kf(j)$, where $\epsilon_1,\epsilon_2\in\{0,1\}$ depending on the remainder of $p$ modulo 12, $k=\left[\frac{p}{12}\right]$, and $f(j)$ is a $p$-integral polynomial w.r.t $j$ of degree $\leq k$. Moreover, the roots of $f(j)$ in $\Fpb$ are simple and exactly the supersingular $j$-invariants excluding 0 and 1728. More precisely, denote by $ss_p(j)=\prod(j-a_i)$ the supersingular polynomial, and then we have
        \[ss_p(j)=\pm j^{\epsilon_1}(j-1728)^{\epsilon_2}f(j)\]
        in $\F_p$. Here, $\{a_i\}$ runs over all of the supersingular $j$-invariants in $\F_{p^2}$.
    \end{enumerate}
\end{prop}
\begin{proof}
    Statement (1) is standard.
    Statement (2) was originally proved by Rademacher in \cite{rademacher1938fourier}. For (3), a concise proof can be found in \cite{kaneko1997supersingular}.
\end{proof}

As pointed out in \cite[Example 1.6.16]{BehrensHandbook}, the normalized Eisenstein series $E_{p-1}$ is the $v_1$-element of $\TMF_p$. And we will see that this decomposition of $E_{p-1}$ splits $\mathrm{Spec}(\pi_0L_{K(2)}\TMF)$ and $\mathrm{Spec}(\pi_0L_{K(1)}L_{K(2)}\TMF)$ into $N_p$ pieces of connected components.

\section{$\G_mL_{K(2)}\TMF$}\label{sec3}
In this section we calculate $\G_mL_{K(2)}\TMF$ via descent---by identifying $L_{K(2)}\TMF$ as the homotopy fixed points of certain Morava $E$-theories, whose spectrum of strict units are known due to work of Hopkins--Lurie, Barthel--Carmeli--Schlank--Yanovski, and Burklund--Schlank--Yuan, the problem reduces to a low-degree group cohomology calculation.

First, recall (cf. \cite{GoerssHopkins}) that given a finite height $n\geq 1$ formal group law $F$ over a perfect field $L$ of characteristic $p$ there exists a 2-periodic $\E_\infty$-ring spectrum $E_n(L)$ called the ``Morava $E$-theory'' (the formal group law $F$ is suppressed from the notation) with $\pi_*E_n(L)\simeq W(L)[[u_1,...,u_{n-1}]]_p[u^\pm]$ where $W(L)$ denotes the $p$-typical Witt vectors, $|u_i|=0$ and $|u|=2$. Moreover, if $K\to L$ is a Galois extension of such fields with Galois group $G$, and the formal group laws are compatible with this extension, then $E_n(K)\simeq E_n(L)^{hG}$. When $L$ is algebraically closed, the strict units in such $\E_\infty$-rings is completely understood:

\begin{thm}[{\cite[Theorem 8.17, Remark 8.18]{BSY}}]\label[thm]{thm:strict_units_of_algebraically_closed_field}
    Let $E_n(L)$ be as above, and suppose that $L$ is algebraically closed. Then we have an equivalence of connective $\Z$-modules
    \begin{equation*}
    \G_mE_n(L)=\Z_p[n+1]\oplus L^\times.
    \end{equation*}
    Moreover, under the $p$-complete equivalence $\mathrm{colim}_kC_{p^k}[n]=C_{p^\infty}[n]\simeq\Z_p[n+1]$, the first summand is the colimit of the \emph{primitive} height $n$ $p^k$-th roots of unity of $E_n(L)$.  
\end{thm}
We will need to calculate $\G_mE(L)^{hG}$ for various groups $G$ acting on $E(L)$ via $\E_\infty$-ring maps. To that end, recall that the space of $\E_\infty$-automorphisms of $E(L)$ is discrete, and is called the \emph{extended Morava stabilizer group} $\G_n$ (the notational clash with $\G_m$ is unfortunate but unavoidable). This group sits in a split exact sequence 
\[\S_n\to\G_n\to\mathrm{Gal}(L/K),\]
where $K$ is the smallest Galois subfield containing the coefficients of $F$. 

The Morava stabilizer group $\G_n$ also admits a preferred homomorphism to $\Z_p^\times$, known as the \emph{cyclotomic character} (cf. \cite[Definition 5.6]{CSY})
\[\chi:\G_n\rightarrow\Z_p^\times.\]

\begin{cor}[{cf. \cite[Remark 8.18]{BSY}}]\label[cor]{cor_action}
 The action of $\G_n$ on $\G_mE_n(L)$ is described as follows: it acts on $L^\times$ via the map to $\mathrm{Gal}(L/K)$ and on $\Z_p[n+1]$ via the cyclotomic character $\chi$.
    
\end{cor}
\begin{proof}
 The first part follows immediately from the fact that the summand $L^\times$ maps injectively to the Teichm\"{u}ller lifts in $W(L)^\times \subset\pi_0E_n(L)^\times$. The second fact is proved as follows. First, consider the ring $E(L)^{hH}$ where $H=\mathrm{ker}\chi$. It is proved in  \cite[Theorem 5.8]{CSY}
 that this agrees with the maximal cyclotomic extension $\S_{K(n)}[\omega_{p^{\infty}}^{(n)}]$ of the $K(n)$-local sphere\footnote{Indeed, this is why $\chi$ contains ``cyclotomic'' in its name.}. This in particular is a summand of the $K(n)$-local group algebra $\S_{K(n)}[B^nC_{p^\infty}]$ and corepresents primitive height $n$ roots of unity. Thus, by \Cref{thm:strict_units_of_algebraically_closed_field} there is a map $\S_{K(n)}[B^nC_{p^\infty}]\to E_n(L)$, adjoint to the inclusion $\Z_p[n+1]\to\G_mE_n(L)$, and furthermore (by the second part of the theorem) it factors through the cyclotomic extension, leading to a diagram
    \[\S_{K(n)}[B^nC_{p^\infty}]\to \S_{K(n)}[\omega_{p^{\infty}}^{(n)}]\simeq E(L)^{hH}\to E(L) \]
    where the canonical $\Z_p[n+1]\subset \G_m\S_{K(n)}[B^nC_{p^\infty}]$ is sent to the $\Z_p[n+1]$ in $\G_mE(L)$. Since $H$ clearly acts trivially on $\G_m$ of the middle factor, it must act trivially on $\G_m$ of the target.
\end{proof}

The restriction of $\chi$ to the subgroup $\mathbb{S}_n$ is already identified in \cite[Theorem 5.8]{CSY} with the determinant on the Dieudonn\'e module. In particular, by \cref{prop:determinant_of_automorphisms}, the automorphism groups described in \cref{prop:automorphism_of_elliptic_curves} have a trivial image under the cyclotomic character $\chi$.

Next, observe that $F$ defines a section of $ \G_n\to\mathrm{Gal}(L/K)$. In order to take fixed points with respect to subgroups of $\mathrm{Gal}(L/K)$, we will need to identify the composite of this section with the cyclotomic character.

\cite[Proposition 5.9]{CSY} provides the necessary description of various Galois actions by virtue of the construction of alternating powers of $p$-divisible groups (cf. \cite{Hedayatzadeh2010exterior} and \cite{lurie2013ambidexterity}). 
For a finite flat group scheme $G$ and every positive integer $d$, there is a finite flat group scheme $\mathrm{Alt}^{(d)}_G$ representing the alternating powers of $G$. In the case of a formal group $\Gamma$ over $\F_{p^m}$, we will use the same notation $\mathrm{Alt}^{(d)}_\Gamma$ to denote the $p$-divisible group $\operatorname{colim}\mathrm{Alt}^{(d)}_{\Gamma[p^n]}$, which is of height $\binom{n}{d}$ and dimension $\binom{n-1}{d}$ if $\Gamma$ is of height $n$. In particular, the base change $\overline{\Gamma}:=\Gamma \times_{\F_{p^m}} \Fpb$, satisfies the property that $\mathrm{Alt}^{(n)}_{\overline{\Gamma}}$ is \'etale of height 1, so it is isomorphic to $\underline{\mathbb{Q}_p/\mathbb{Z}_p}$ with the automorphism group $\Z^\times_p$. Consequently, $\mathrm{Alt}^{(n)}_\Gamma$ determines a class in $H^1_c(\mathrm{Gal}(\Fpb/\F_p),\mathrm{Aut}(\underline{\mathbb{Q}_p/\mathbb{Z}_p}))$.

\begin{prop}[{cf. \cite[Proposition 5.9]{CSY}}]
    Fix a formal group $\Gamma$ of height $n$ over $\F_{p^m}$, which induces a section $\sigma:\hat{\Z} \to \G_n$. Then the $p$-divisible group $\mathrm{Alt}^{(n)}_\Gamma$ can be identified with the restriction character $\hat{\Z}\xrightarrow{\sigma} \G_n\xrightarrow{\chi}\Z^\times_p$ under the isomorphism
    \[ H^1_c(\mathrm{Gal}(\Fpb/\F_{p^m}),\mathrm{Aut}(\underline{\mathbb{Q}_p/\mathbb{Z}_p}))\simeq \mathrm{Hom}_c(\mathrm{Gal}(\Fpb/\F_{p^m}),\Z^\times_p).\]
\end{prop}

\begin{defn}[{cf. \cite[Definition 5.3.1]{lurie2013ambidexterity}}]\label[defn]{defn_normalizable}
    Let $\Gamma$ be a formal group of height $n$ over the finite field $\F_{p^m}$. Say $\Gamma$ is normalizable if there is an isomorphism of $p$-divisible groups over $\F_{p^m}$
    \[\mathrm{Alt}^{(n)}_\Gamma\simeq \underline{\Q_p/\Z_p}.\]
\end{defn}

For instance, the formal completion of supersingular elliptic curves gives rise to a large family of normalizable formal groups.

\begin{prop}\label[prop]{prop_normlizable}
    Let $E$ be a supersingular elliptic curve over a finite field $\F_{p^m}$. The alternating power $\mathrm{Alt}^{(2)}_{\hat{E}}$ is isomorphic to $\underline{\mathbb{Q}_p/\mathbb{Z}_p}$ witnessed by Weil pairing. In particular, the formal group $\hat{E}$ is normalizable, and in $H^1_c(\mathrm{Gal}(\Fpb/\F_{p^m}),\mathrm{Aut}(\underline{\mathbb{Q}_p/\mathbb{Z}_p}))$, the class  classifying $\mathrm{Alt}^{(2)}_{\hat{E}}$ is zero.
\end{prop}
\begin{proof}
    We need to construct a compatible family of group homomorphisms \[\underline{\mathbb{Z}/p^r}\to \mathrm{Alt}^{(2)}_{E[p^r]},\] and show that it is an isomorphism. By definition, the data of this map is the same as the alternating bilinear pairing $E[p^r]\times E[p^r]\to\mu_{p^r}\subset \G_m$. The Weil pairing provides reasonable candidates, as it is well known to be compatible with multiplication by $p$ and with restriction maps. Moreover, after base change to $\Fpb$, the sheaf $\mathrm{Alt}^{(2)}_{E[p^r]}$ becomes the constant group $\underline{\Z/p^r}$. The nondegeneracy of Weil pairing implies that the resulting maps over $\Fpb$ are injective endomorphisms of $\underline{\mathbb{Z}/p^r}$ over $\Fpb$, hence automatically isomorphisms.
    Therefore, by faithful descent from $\Fpb$ to $\F_{p^m}$, the original map $\underline{\mathbb{Z}/p^r}\to \mathrm{Alt}^{(2)}_{E[p^r]}$ is also an isomorphism.
\end{proof}
\begin{cor}\label[cor]{cor_gal_action}
    Let $E_2(\Fpb)$ be the Morava $E$-theory with FGL given by the formal completion of a supersingular elliptic curve over $\F_{p^m}$. The action of $\mathrm{Gal}(\Fpb/\F_{p^m})$ on $\Z_p[3]\subset \G_mE_2(\Fpb)$ is trivial.
\end{cor}
Finally, we have the following proposition.

\begin{prop}[{cf. \cite[Section 1.6]{BehrensHandbook}}]
\label[prop]{K(2)_TMF_formula}
We have an equivalence of $\mathbb{E}_\infty$-rings
$L_{K(2)}\TMF\simeq\left(\prod_{i\in S_p}E_2(\Fpb)^{hG_i}\right)^{h\hat{\Z}},$ where $\hat{\Z}$ acts on $E_2(\Fpb)$ via the section $\sigma_i:\hat{\Z}\to\G_n$ determined by the formal group law associated to the elliptic curve indexed by $i$, and on $S_p$ by its natural permutation action.
\end{prop}
\begin{proof}
     This is contained in \cite[Section 1.6]{BehrensHandbook}, though the final formula is only quoted there for $p=2,3$. By \cite[Proposition 1.6.8]{BehrensHandbook}, $L_{K(2)}\TMF$ is the global sections of the restriction of $\mathcal{O}^\mathrm{top}_{\mathcal{M}_\mathrm{ell}}$ to $\widehat{\mathcal{M}^\mathrm{ss}_\mathrm{ell}}$, the formal neighborhood of the moduli stack of supersingular elliptic curves $\mathcal{M}^\mathrm{ss}_\mathrm{ell}\hookrightarrow\mathcal{M}_\mathrm{ell}$.
     After base change to $\Fpb$, $\mathcal{M}^\mathrm{ss}_\mathrm{ell}$ is a disjoint union over a finite set $S_p$ (the set of isomorphism classes of supersingular elliptic curves over $\Fpb$) of groupoids $\pt/G_i$ where $G_i$ is the automorphism group of some elliptic curve representative $E_i$ of $i\in S_p$. Thus by the Serre--Tate theorem, we see that
     \[L_{K(2)}\TMF\to\prod_{S_p}E_2(\Fpb)\]
     is a $\hat{\Z}$-Galois cover of a disjoint union of $G_i$-Galois covers, giving the desired formula.
\end{proof}

\subsection{$p=2$}
\begin{prop}
   There is an equivalence of connective $\Z$-modules
    \begin{equation*}
        \G_mL_{K(2)}\TMF=\Z_2[2]\oplus\Z_2[3].
    \end{equation*}
\end{prop}
\begin{proof}
     \cref{example:supersingular_elliptic_curves_at_p=2} and \cref{prop:automorphism_of_elliptic_curves} imply that there is only one supersingular elliptic curve, which is already defined over $\F_2$, and its $\overline{\F}_2$-automorphism group $G=G_{24}=\Z/3\ltimes Q_8$. Thus from \Cref{K(2)_TMF_formula}, we see that
\[L_{K(2)}\TMF\simeq\left(E_2(\overline{\F}_2)^{hG_{24}}\right)^{h\hat{\Z}}.\]
By \Cref{prop:determinant_of_automorphisms} and \Cref{cor_action}, the action of $G_{24}$ on $\G_mE_2(\Fpb)=\Z_2[3]\oplus\overline{\F}_2^\times$ is trivial. So we get
\[\G_mE_2(\Fpb)^{hG_{24}}\simeq \Z_2[3]^{BG_{24}}\simeq \Z_2[3]\]
by group cohomology computation.
It follows from \Cref{cor_gal_action} that the action of $\hat{\Z}$ is also trivial.
Using the 2-adic equivalences $B\hat{\Z}\simeq B\Z_2\simeq S^1$, we get
\[\G_mL_{K(2)}\TMF=\left(\G_mE_2(\Fpb)^{BG_{24}}\right)^{B\hat{\Z}}=\Z_2[2]\oplus \Z_2[3]\]

as desired.
\end{proof}
\subsection{$p=3$}
By \Cref{prop:automorphism_of_elliptic_curves} there is only one supersingular elliptic curve with stabilizer $C_4\ltimes C_3$.

\begin{prop}
    We have an equivalence between connective $\Z$-modules
    \begin{equation*}
        \G_mL_{K(2)}\TMF=\F_3^\times\oplus\Z_3[2]\oplus\Z_3[3].
    \end{equation*}
\end{prop}
\begin{proof}
    The argument is identical to the one at $p=2$, with $G_{24}$ replaced by $C_4\ltimes C_3$. Thus the final answer reads
    \[\G_mL_{K(2)}\TMF\simeq \left(\Z_3[3]\oplus\Z_3[2]\oplus \F_3^\times\right)^{BC_4\ltimes C_3}\]
   whose connective cover is $\F_3^\times\oplus\Z_3[2]\oplus\Z_3[3]$ by group cohomology of dihedral groups as desired.
\end{proof}

\subsection{$p\geq 5$}
\begin{prop}
    Denoting $N_p$ by the number of orbits appearing in the Galois action of $S_p$ as in \cref{notation:Np}, we have an equivalence\[\G_mL_{K(2)}\TMF=\left(\Z_p[3]\oplus\Z_p[2]\right)^{\oplus N_p}\oplus (\F_p^\times)^{\oplus(2N_p-\#S_p) }\oplus(\F_{p^2}^\times)^{\oplus(\#S_p-N_p)}.\]
\end{prop}
\begin{proof}
    By \Cref{K(2)_TMF_formula}, we have
\[\G_mL_{K(2)}\TMF\simeq\left(\prod_{i\in S_p}\G_mE_2(\Fpb)^{hG_i}\right)^{h\hat{\Z}},\]

where $S_p$ is the collection of isomorphism classes of supersingular curves over $\Fpb$, and $G_i$ is the automorphism group of that curve. By \cref{prop:automorphism_of_elliptic_curves} $G_i$ is either $\Z/2$, $\Z/4,$ or $\Z/6$, and by \cref{prop:determinant_of_automorphisms} they act trivially. As these groups are all coprime to $p$, the formula reduces to
\[\G_mL_{K(2)}\TMF\simeq\left(\prod_{i\in S_p}\G_mE_2(\Fpb)\right)^{h\hat{\Z}}.\]
Applying the formula in \Cref{thm:strict_units_of_algebraically_closed_field} to $\G_mE_2(\Fpb)$, we need to compute the homotopy fixed points for each degree. According to \Cref{cor_gal_action}, the Galois group $\hat{\Z}$ acts on $(\Z_p)^{\oplus \#S_p}$ solely by permuting the indexing set, so we have the summand $(\Z_p[3]\oplus\Z_p[2])^{N_p}$. For $(\Fpb^\times)^{\oplus \#S_p}$, the Galois group acts both via Frobenius on each summand and on the indexing set. In particular, for an orbit of order 2, it maps a pair $(a,b)$ to $(b^p,a^p)$, whence the fixed points are of form $(a,a^p)$ with $a\in\F_{p^2}^\times$. This provides the summands $(\F_p^\times)^{\oplus(2N_p-\#S_p) }\oplus(\F_{p^2}^\times)^{\oplus(\#S_p-N_p)}$.
\end{proof}
\section{Facts about $K(1)$-local power operations}\label{sec4}
In this section we collect some facts about $K(1)$-local power operations which we will need below. The first fact is that for any $K(1)$-local $\E_\infty$-ring $R$, there exists a canonical power operation, given by a map of spaces $\theta:\Omega^\infty R\rightarrow R$, and such that the free $K(1)$-local $\E_\infty$-ring on one generator $x$ of degree $0$ has homotopy groups given by the infinite polynomial algebra
\[\pi_*\mathrm{Free}_{K(1)}(x)\simeq\pi_*\S_{K(1)}[x, \theta(x), \theta^2(x), ...].\]
It follows that the horizontal leg in the pushout diagram
\[\begin{tikzcd}
    \mathrm{Free}_{K(1)}(x)\arrow[r,"x\mapsto\theta(y)"]\arrow[d,"x\mapsto0"]&\mathrm{Free}_{K(1)}(y) \\
    \S_{K(1)}& \\
\end{tikzcd}\]
is flat, and hence this pushout presents the spherical monoid algebra $\S_\K[\mathbb{N}]$ of the natural numbers. Hence, inverting $y$ in the top right node is again a pushout diagram, this time presenting the spherical group algebra $\S_\K[\mathbb{Z}]$. Given the adjunction $\S[-]:\Sp^\cn\rightleftharpoons\calg: \glone$, we thus get a fiber sequence of spaces
\[\G_m(R)\rightarrow \mathrm{GL}_1(R)\oto{}\Omega^\infty R.\]
We call the corresponding long exact sequence of homotopy groups the $\theta$-LES, and we write $\theta_n$ for the map $\pi_n\GL_1(R)\rightarrow\pi_n\Omega^\infty R$. While $\theta_0$ coincides with the restriction of $\theta$ to the units in $\pi_0R$ by definition, to calculate the higher $\theta_n$'s, we have the following result.

\begin{lem}\label[lem]{lem_theta_n}
    Let $R$ be a $K(1)$-local $\E_\infty$-ring at the prime 2. Write $\epsilon_n\in R^n(S^n)$ for the canonical class. Observe that $R^{S^n}$ is also a $K(1)$-local $\E_\infty$-ring, and thus also has a canonical $\theta$-operation. For a class $\alpha\in R^k(S^n)$ define $\l \alpha, 1\r\in \pi_{-k}R$ and $\l \alpha, \epsilon_n\r\in \pi_{n-k}R$ by the formula
    \[\alpha = \l \alpha, 1\r1+\l \alpha, \epsilon_n\r\epsilon_n.\]
    Then for any $\xi\in\pi_nR$ we have
    \[\theta_n(\xi)=\l\theta_0(\xi\epsilon_n),\epsilon_n\r-\xi.\]
\end{lem}
\begin{proof}
    We have the following chain of equalities:
    \[\theta_n(\xi)=\l\theta_0(1+\xi\epsilon_n),\epsilon_n \r=\l\theta_0(1)+\theta_0(\xi\epsilon_n)-\xi\epsilon_n,\epsilon_n \r=\l\theta_0(\xi\epsilon_n),\epsilon_n\r-\xi.\]
    Here, the first equality follows by naturality of $\theta$ (the composite of $R$'s $\theta$ and a map $\S^n\to \glone R$ is equal to $R^{S^n}$'s $\theta$ on the class associated to the map) together with the fact that the basepoint of $\GL_1 R$ is 1, while that of $\Omega^\infty R$ is 0. The second equality follows from the fact that $\theta$ is a $\delta$-structure and thus satisfies the standard $\delta$-ring identities for $\delta(xy)$ and $\delta(x+y)$.
\end{proof}

Before we move to the $K(1)$-local computations, let us recall a lemma in \cite{CarmeliLuecke}, which ensures that it's reasonable to consider only the first several terms in $\theta$-LES.

\begin{lem}[{cf. \cite[Lemma 3.1.5]{CarmeliLuecke}}]\label[lem]{truncated}
    For any $T(1)$-local ring $R$, $\G_m(R)$ is 2-truncated, and $\pi_2\G_m(R)$ is torsion-free.
\end{lem}

\section{$\G_mL_{K(1)}\TMF$}\label{sec5}
In this section, we use the results of the previous section to calculate $\G_m$ of the $K(1)$-localizations of $\TMF$, and $\tmf$.
\subsection{$p=2$}

\begin{prop}
There is an equivalence of connective $\Z$-modules
    \[\G_mL_{K(1)}\TMF=\Z/2 [1]\oplus(\Z/2)^\infty.\]
    Moreover, the map to $\glone L_{K(1)}\TMF$ detects $\eta$ via the first summand and is zero on the second.
\end{prop}
\begin{proof}
    By e.g., \cite[Corollary 3]{Laures}\footnote{It is worth remarking that this isomorphism comes from an isomorphism of spectra, but not of ring spectra (cf. e.g., \cite[Remark 2]{Laures}), as the source is not a $\KO$-algebra.},
we have
\[\pi_*L_{K(1)}\tmf\simeq\pi_*\KO[j^{-1}]_2.\]
Since $j=c_4^3\Delta^{-1}$, inverting $\Delta^{24}$ is the same as inverting $j^{-1}$, and we get
\[\pi_*L_{K(1)}\TMF\simeq\pi_*\KO[j^{\pm}]_2.\]
The $\theta$-LES becomes
\[
\begin{tikzcd}
     \pi_2\G_mL_{K(1)}\TMF\arrow[r]&\F_2 [j^{\pm}]\eta^2 \arrow[r,"\theta_2"]& \F_2 [j^{\pm}]\eta^2\\
    \pi_1\G_mL_{K(1)}\TMF \arrow[r]&\F_2 [j^{\pm}]\eta \arrow[r,"\theta_1"]&\F_2 [j^{\pm}]\eta \\
     \pi_0\G_mL_{K(1)}\TMF\arrow[r]&(\Z[j^{\pm}])_2^\times \arrow[r,"\theta_0"]&\Z[j^{\pm}]_2. \\
\end{tikzcd}
\]
To calculate the $\theta_i$, we proceed as follows:

\begin{itemize}
    \item Consider the classical $q$-expansion map
    \begin{equation*}
\Z[j^{\pm}]_2\rightarrow \Z((q))_2
    \end{equation*}
    sending $j$ to the famous
    \[j=q^{-1}+744+196884q+\dots\]
    as recorded in \cref{prop_modular_form_identities}. By \cite{davies2025derived}, $q$-expansion admits an
 $\mathbb{E}_\infty$-enhancement---that is, there is an $\E_\infty$-map  $\TMF\rightarrow \KO((q))$ which on $\pi_0$ (and after $K(1)$-localization) recovers the classical $q$-expansion map. Therefore, the source and target acquire canonical $\theta$-operations, and $q$-expansion is compatible with these. Moreover, $\theta$ on the target is determined by $\theta(q)=0$. In particular, its kernel consists of the monomials in $q$. By a transcendence argument \cite{voloch1996transcendence}, any element of the source which is killed by $\theta_0$ must be constant in $j$. On the constants, $\theta_0$ is the unique $\delta$ structure on  $\Z_2$, and its kernel is $\F_2^\times\simeq 0$. 
    
    \item To calculate $\theta_1$ we proceed as follows: write an element in $\pi_1$ as $f(j)\eta$ where $f(j)\in\Z/2[j^\pm]$.
    Then by \Cref{lem_theta_n} we have
    \[\theta_1(f(j)\eta)=\l\theta_0( f(j)\eta\epsilon_1),\epsilon_1\r.\]
    Using the formula for $\theta(xy)$ with $x=f(j)$ and $y-\theta\epsilon_1$, this becomes
    \[\l f(j)^2\theta_0(\eta\epsilon_1)+\theta_0(f(j))(\eta\epsilon_1)^2,\epsilon_1\r-f(j)\eta.\]
    Finally, using $\theta(\eta\epsilon_1)=\eta\epsilon_1$ (which can be checked in $\KO_2$ where $\theta_1(\eta)=0$, cf. e.g., \cite{CarmeliYuan}) this becomes
    \[=(f(j)^2-f(j))\eta.\]
    Thus $\theta_1$ is injective except on the constant polynomials (i.e., on $\eta$), and has a cokernel generated by $\eta$ and the odd powers $j^{2\Z+1}\eta$ (in the cokernel, $j^{2n}\eta$ is equivalent to $j^n\eta$).
    \item  Since $\pi_3\KO[j^{\pm}]=0$,  $\pi_2\G_mL_{K(1)}\TMF$ is a subgroup of $\F_2[j^\pm]\eta^2$ and $\pi_2\G_mL_{K(1)}\TMF$ is torsion-free for general reasons (see \cref{truncated}), we see that $\pi_2\G_mL_{K(1)}\TMF$ is zero, and $\theta_2$ is injective. To show that it is also surjective, we calculate it explicitly, using \Cref{lem_theta_n} and $\delta$-ring identities to obtain
    \[\theta_2(f(j)\eta^2)=\l f(j)^2\theta_0(\eta^2\epsilon_2)+\theta_0(f(j))(\eta^2\epsilon_2)^2-f(j)\eta^2\epsilon_2,\epsilon_2\r.\]
    Using $\theta_0(\eta^2\epsilon_2)=0$ (which again can be checked in $\KO_2$ where $\theta_2(\eta^2)=\eta^2$) and $\eta^4=0$, this becomes
    \[=-f(j)\eta^2.\]
\end{itemize}
So we have the equivalence
\[\G_mL_{K(1)}\TMF=\Z/2 [1]\oplus(\Z/2)^\infty\]
by combining the kernels and cokernels of $\theta_i$.
\end{proof}
\begin{rem}
    As is evident from the first display of the proof, the analogous calculation for $\tmf$ is almost identical---the only change is that the point $j=\infty$ has been added back in, so one replaces $\Z[j^{\pm}]_2$ with $\Z[j^{-1}]_2$.
\end{rem}
\begin{cor}
    There is an equivalence of connective $\Z$-modules
    \[\G_mL_{K(1)}\tmf=\Z/2 [1]\oplus(\Z/2)^\infty.\]
    Moreover, the map to $\glone L_{K(1)}\tmf$ detects $\eta$ with the first summand and is zero on the second.
\end{cor}
\begin{rem}\label[rem]{rem_tmf_TMF_K(1)}
    The map $\tmf\to\TMF$ does not induce an isomorphism on $\G_mL_{K(1)}$---in $\pi_0$, the $(\Z/2)^\infty$ in the source is indexed by $\mathbb{N}$ while in the target it is indexed by $\mathbb{Z}$.
\end{rem}

\subsection{$p=3$}
We have (cf. \cite[Chapter 12.1]{TMFbook})
\[\pi_*L_{K(1)}\tmf\simeq\pi_*\KO[j^{-1}]_3,\]
and thus
\[\pi_*L_{K(1)}\TMF\simeq\pi_*\KO[j^{\pm}]_3,\]
which is 4-concentrated. Thus, via the $\theta$-LES, the corresponding strict units spectrum $\G_m$ is concentrated in $\pi_0$ and is equal to the kernel of $\theta_0$. By the same argument as at $p=2$, the kernel of $\theta_0$ is containted in the constant functions $\Z_3$, where $\theta_0$ is the unique $\delta$-structure on $\mathbb{Z}_3$, whose kernel is $\F_3^\times$. To conclude, we have
\begin{prop}
There is an equivalence of connective $\Z$-modules
    \begin{equation*}
\G_mL_{K(1)}\TMF=\F_3^\times,
    \end{equation*}
    and the map to $\glone L_{K(1)}\TMF$ is injective onto the Teichm\"{u}ller lifts $\F_3^\times\subset\Z_{3}\subset \pi_0L_{K(1)}\TMF$.
\end{prop}

\subsection{$p\geq 5$}
Similar to the previous cases, we have the homotopy groups of $L_{K(1)}\tmf$ (cf. \cite[Chapter 12.1]{TMFbook}) 
\[\pi_*L_{K(1)}\tmf=\mathcal{O}\left(\P^1_\Z\setminus S_p\right)_p[b^\pm],\]
where $|b|=4$ and $S_p$ is an integral lift of the $\F_p$-scheme of supersingular $j$-values. Again, since $j=c_4^3\Delta^{-1}$, inverting $\Delta$ deletes the point at $\infty$ in $\P^1$, i.e., inverts $j^{-1}$. Hence, we get
\[\pi_*L_{K(1)}\TMF=\Z[j,(ss_p(j))^{-1}]_p[b^{\pm}],\]
 and $ss_p(j)$ means an integral lift of the supersingular polynomial at $p$ (see \cref{prop_modular_form_identities}). 
 The homotopy groups are still 4-concentrated, so by $\theta$-LES, $\G_m$ is concentrated in $\pi_0$ and equal to the kernel of $\theta_0$. Just as at the previous primes, this kernel is contained in the constant functions $\mathbb{Z}_p$, and coincides with $\mathbb{F}_p^\times$.
 \begin{prop}
There is an equivalence of connective $\Z$-modules
    \begin{equation*}
\G_mL_{K(1)}\TMF=\F_p^\times,
    \end{equation*}
    where all the elements are mapped injectively into $\pi_0L_{K(1)}\TMF$ through Teichm\"{u}ller lifts.
\end{prop}

\begin{rem}\label[rem]{delta}
    In the proof above, we reduce the analysis of $\theta_0(j)$ to the $\delta$ structure of $\KU((q))$. The same strategy allows us to inductively compute $\theta_0(j)$ in $\pi_0L_{K(1)}\TMF$ using a series in $j$:
    Since each ring under consideration is $p$-torsionfree, the $\delta$ structure induced by $\theta_0$ is equivalent to the lift of Frobenius, given by
\[\varphi(x):=x^p+p\theta_0(x).\]
In $\Z((q))_p$ this Frobenius lift sends $j(q)$ to $j(q^p)$, as $\delta(q)=0$. Consequently, the preimage of $j(q^p)$ under $q$-expansion is precisely our target $\varphi(j)$.
On the other hand, classical modular form theory provides the \emph{modular polynomials} $\Phi_p(x,y)\in\Z[x,y]$, which satisfies $\Phi_p(j(q),j(q^p))=0$. Although the explicit coefficients of $\Phi_p$ are rather complicated (cf. \cite{broker2012modular}), the classical Kronecker's theorem gives the congruence relation
\[\Phi_p(x,y)\equiv (y^p-x)(y-x^p)\quad(\mathrm{mod}\,p).\]
Hence, using Hensel's lemma in $\pi_0L_{K(1)}\TMF\simeq\Z[j,(ss_p(j))^{-1}]_p$, we obtain the unique factorization of $\Phi_p(j,y)$, where the factor $y-j^p$ lifts uniquely to $y-\varphi(j)$ by comparing the power $q^{-p}$.

The complexity of computing $\theta_0(j)$ is sensitive to different primes $p$, because the supersingular polynomial may appear in the denominator and the modular polynomial is far beyond control. The authors computed the first several expansions for $p=2$ and $p=3$:
\begin{align*}
    &\varphi_2(j)\equiv j^2+16j\quad(\mathrm{mod}\,32),\\
    &\varphi_3(j)\equiv j^3+9j^2\quad(\mathrm{mod}\, 27).
\end{align*}
\end{rem}

\section{$\G_mL_{K(1)}L_{K(2)}\TMF$}\label{sec6}
In this section we analyze $\G_m$ of the remaining node of the chromatic fracture square. At $p=2,3$, the key is to use the map $L_{K(1)}\TMF\rightarrow L_{K(1)}L_{K(2)}\TMF$ to identify the target as a completion of the source, and then carry over calculations done in previous sections. At large primes, we instead rely on the sparseness of homotopy groups to make the required calculation.

The only issue for $L_{K(1)}L_{K(2)}\TMF$ is that the $\mathbb{E}_\infty$-enhancement of $q$-expansion does not induce a nontrivial map $L_{K(1)}L_{K(2)}\TMF\to\KU((q))$, and we must compute the kernel of $\theta_0$ by hand. This task is not accessible without an explicit understanding of $\delta(j)$. Nevertheless, we are able to identify certain subgroups of $\pi_0\G_mL_{K(1)}L_{K(2)}\TMF$ that suffices for our purpose. Indeed, because this $\pi_0$ enters only at the terminal end of the relevant long exact sequence, a subgroup is already sufficient to establish the main theorem of computing the strict units of $\TMF$.
\subsection{$p=2$}
 The fiber of $L_{K(1)}\TMF\rightarrow L_{K(1)}L_{K(2)}\TMF$ is isomorphic to that of $\TMF_2\rightarrow L_{K(2)}\TMF$ by chromatic fracture. Moreover, (cf. \cite[Proposition 1.6.14]{BehrensHandbook})
\[\pi_*L_{K(2)}\TMF=\pi_*\TMF_{(2,c_4)}\]
 and the map $\TMF_2\rightarrow L_{K(2)}\TMF$ is the $c_4$-completion map. By the formula $j=c_4^3\Delta^{-1}$, $c_4$-completion coincides with $j$-completion. Furthermore, by \cite[Proposition 1.6.8]{BehrensHandbook}, we have 
 $\pi_*L_{K(1)}L_{K(2)}\TMF\simeq \pi_*\TMF_{(2,c_4)}[c_4^{-1}]_2.$
 Combining these, we find that the map \[\pi_*L_{K(1)}\TMF\rightarrow \pi_*L_{K(1)}L_{K(2)}\TMF\]
 exhibits the target as the ``$j$-completion'' of the source---that is, the homotopy groups of the target coincide with those of the former but with $\Z[j^{\pm}]_2$ replaced by $\Z((j^{}))_2$. 
So the $\theta$-LES is
\[
\begin{tikzcd}
     \pi_2\G_mL_{K(1)}L_{K(2)}\TMF\arrow[r]&\F_2 ((j))\eta^2 \arrow[r,"\theta_2"]& \F_2 ((j))\eta^2\\
    \pi_1\G_mL_{K(1)}L_{K(2)}\TMF \arrow[r]&\F_2 ((j))\eta \arrow[r,"\theta_1"]&\F_2 ((j))\eta \\
     \pi_0\G_mL_{K(1)}L_{K(2)}\TMF\arrow[r]&(\Z((j)))_2^\times \arrow[r,"\theta_0"]&\Z((j))_2. \\
\end{tikzcd}
\]

 The $\theta$-structure is determined by the one on $\Z[j^{\pm}]_2$, and so the formulas for $\theta_i$ are still valid, so the calculation of $\G_mL_{K(1)}L_{K(2)}\TMF$ is thus similar to that of $\G_mL_{K(1)}\TMF$ with the change from Laurent polynomials to Laurent series, giving rise to the following two discrepancies:
 \begin{enumerate}
     \item The calculation in the cokernel of $\theta_1$ differs. Positive powers $j^n$ are now in the image of $\theta_1$, detected by the element $j^n+j^{2n}+...$, so the cokernel is spanned by $j^0$ and the odd negative powers of $j$.
     \item The kernel of $\theta_0$ acting on $\Z((j))_2^\times$ is more tedious to determine. However, $\mathrm{Ker}(\theta_0)$ is not important for the global calculation of $\G_m\TMF$, since the composition
\[\pi_0\G_mL_{K(1)}\TMF\to\pi_0\G_mL_{K(1)}L_{K(2)}\TMF\twoheadrightarrow\mathrm{Ker}(\theta_0)\]
     is always 0 by diagram chasing.
 \end{enumerate}
Hence, we have


\begin{prop}\label{LK1LK2_at_p=2}
    There is an equivalence of connective $\mathbb{Z}$-modules 
    \[\mathrm{fib}(\G_mL_{K(1)}L_{K(2)}\TMF\to\mathrm{Ker}(\theta_0))=\Z/2 [1]\oplus(\Z/2)^\infty.\]
    Moreover, the forgetful map from this fiber to $\glone L_{K(1)}L_{K(2)}\TMF$ detects $\eta$ with the first summand and is zero on the second, and the map
    \[\G_mL_{K(1)}\TMF=\Z/2 [1]\oplus(\Z/2)^\infty\to\G_mL_{K(1)}L_{K(2)}\TMF.\]
    is an isomorphism in degree 1 and a surjection on $\pi_0$ with kernel isomorphic to $(\Z/2)^\infty$.
\end{prop}
\begin{rem}
    In juxtaposition to the second part of the above proposition, we see that the map $\G_m L_{K(1)}\tmf \to \G_mL_{K(1)}L_{K(2)}\TMF$ is an isomorphism if $\mathrm{Ker}(\theta_0)$ vanishes.
\end{rem}
 

\subsection{$p=3$}
The strategy is identical to the case $p=2$, but simpler, as the homotopy groups become sparser: we again have (cf. \cite[Proposition 1.6.14]{BehrensHandbook})
\[\pi_*L_{K(2)}\TMF=\pi_*\TMF_{(3,c_4)}\]
and find that $\pi_*L_{K(1)}L_{K(2)}\TMF$ is 4-concentrated with $\pi_0L_{K(1)}L_{K(2)}\TMF\simeq \Z((j))_3$. Thus, by $\theta$-LES, we partially determine the structure of $\G_mL_{K(1)}L_{K(2)}\TMF$.
\begin{prop}
The strict units spectrum of $L_{K(1)}L_{K(2)}\TMF$ is discrete and contains a subgroup $\F_3^\times$ in $\pi_0$.
Besides, the forgetful map from this subgroup to $\Z_3^\times\subset\pi_0\glone L_{K(1)}L_{K(2)}\TMF$ injects onto the Teichm\"{u}ller lifts.
\end{prop}

\subsection{$p\geq 5$}
We use the same approach as $p=2,3$, but the completion formula is slightly different, as the relevant completion is not controlled by $c_4$ but rather a different modular form, with more than one zero.

Indeed, we have (cf. \cite[Rem 1.6.16]{BehrensHandbook})
$\pi_*L_{K(2)}\TMF=\pi_*\TMF_{(p,E_{p-1})}$
for $E_{p-1}$ the normalized Eisenstein modular form. Due to \cref{prop_modular_form_identities}, $E_{p-1}$ is equal to $c_4^{\epsilon_1}c_6^{\epsilon_2}\Delta^kf(j)$ in $\pi_{2p-2}\TMF_p$, and the zeros of $f(j)$ coincide with the supersingular $j$-values $a_i$ excluding 0 and 1728. At $\pi_0$, this completion is thus the same as $\mathbb{Z}[j]_{(p,ss_p(j))}$, where $ss_p(j)$ is an integral lift of the supersingular polynomial, i.e., the product of $(j-a_i)$'s. It follows from the action of $\mathrm{Gal}(\F_{p^2}/\F_p)$ on the isomorphism classes of supersingular elliptic curves $S_p$ that, modulo $p$, the polynomial $ss_p(j)$ can be decomposed into the product
\[ss_p(j)=\prod (j-a_i)\times\prod (j^2-b_kj+c_k).\]
Here, the first factor consists of supersingular $j$-invariants in $\F_p$. And the second product runs over all orbits $(a_i,a_{i'})$ of order 2---that is, $a_i,a_{i'}\in \F_{p^2}\backslash\F_p$ are conjugate to each other via Frobenius, $b_k:=a_i+a_{i'}$ and $c_k:=a_ia_{i'}$ are thus elements in $\F_p$. There are in total $N_p$ coprime factors appearing in $ss_p(j)$, and by Hensel's lemma, such a factorization lifts to $\Z_p$.
So we find that
\[\pi_*L_{K(1)}L_{K(2)}\TMF=\left(\prod\Z((j-a_i))_p \times \prod \Z[j]^\wedge_{(p,j^2-b_kj+c_k)}[(j^2-b_kj+c_k)^{-1}]_p\right)[b^\pm].\]
\begin{war}
    The residue field of $\Z[j]^\wedge_{(p,j^2-b_kj+c_k)}$ is not $\F_p$ but $\F_{p^2}$, since it is equal to $\F_p[j]/(j^2-b_kj+c_k)$. In particular, there are the Teichm\"{u}ller lifts $\F_{p^2}^\times\subset W(\F_{p^2})$ whose image of $\delta$ vanishes.
\end{war}
By construction, modulo $p$, the Frobenius homomorphism fixes each factor, inducing an isomorphism of $\delta$-rings for $\pi_0L_{K(1)}L_{K(2)}\TMF$.
Again, due to the 4-concentrated property, the $\theta$-LES is only possibly non-trivial in $\pi_0$, and $\pi_0\G_m$ is exactly $\mathrm{Ker}(\theta_0)$, where each connected component (of the corresponding scheme) contributes a factor of $\F_p^\times$ or $\F_{p^2}^\times$. Thus, regardless of the other Laurent series in $\mathrm{Ker}(\theta_0)$, we find that

 \begin{prop}
The strict units spectrum of $L_{K(1)}L_{K(2)}\TMF$ is discrete and contains a subgroup $$(\F_p^\times)^{\oplus(2N_p-\#S_p) }\oplus(\F_{p^2}^\times)^{\oplus(\#S_p-N_p)}$$ in $\pi_0$.
Besides, the map from this subgroup to $\pi_0\glone L_{K(1)}L_{K(2)}\TMF$ injects onto the Teichm\"{u}ller lifts described above.
\end{prop}
\begin{ques}
    Does $\mathrm{Ker}(\theta_0)$ consist of these Teichm\"{u}ller lifts? In other words, does the subgroup identified in this section exhaust
    $\pi_0\G_mL_{K(1)}L_{K(2)}\TMF$?
\end{ques}
The authors have verified the cases $p=2$ or $p=3$, based on the first several estimations recorded in \cref{delta}. However, a generalization to arbitrary prime numbers is considerably difficult due to the complexity of both supersingular polynomials and modular polynomials. Given that what we compute is sufficient for the main theorem, we defer a systematic treatment of the general case to future work.

\section{Assembling the fracture squares}\label{sec7}
In this section we use the chromatic and arithmetic fracture squares to assemble the information gathered above into a proof of our main theorem.

\subsection{Chromatic fracture}
Since both $\glone(-)$ and $\G_m(-)=\Map(\Z,\glone(-))$ are right adjoints, they commute with limits. Applying them to the chromatic fracture pullback square for $\TMF_p$ is thus again a pullback square:
\[
\begin{tikzcd}
    \G_m\TMF_p\arrow[r]\arrow[d]&\arrow[d]\G_mL_{K(2)}\TMF \\
\G_mL_{K(1)}\TMF\arrow[r]&\G_mL_{K(1)}L_{K(2)}\TMF.\end{tikzcd}\]
Since $\G_m$ of a $p$-complete, $L_2$-local ring is 3-truncated, we may restrict attention to the following ``critical range'' of degrees, where the associated LES looks like
\[
\begin{tikzcd}
    \pi_3\G_m\TMF_p\arrow[r,"\sim"]&\pi_3\G_mL_{K(2)}\TMF \arrow[r]& 0 \\
\pi_2\G_m\TMF_p\arrow[r,hook]&\pi_2\G_mL_{K(1)}\TMF\oplus\pi_2\G_mL_{K(2)}\TMF \arrow[r]& \pi_2\G_mL_{K(1)}L_{K(2)}\TMF \\
\pi_1\G_m\TMF_p\arrow[r]&\pi_1\G_mL_{K(1)}\TMF\oplus\pi_1\G_mL_{K(2)}\TMF \arrow[r]& \pi_1\G_mL_{K(1)}L_{K(2)}\TMF \\
\pi_0\G_m\TMF_p\arrow[r]&\pi_0\G_mL_{K(1)}\TMF\oplus\pi_0\G_mL_{K(2)}\TMF \arrow[r]& \pi_0\G_mL_{K(1)}L_{K(2)}\TMF.\\
\end{tikzcd}
\]
Filling in with information from previous sections, we get the following:
\subsubsection{$p=2$}
\[
\begin{tikzcd}
    \pi_3\G_m\TMF_2\arrow[r,"\sim"]&\Z_2 \arrow[r]& 0 \\
\pi_2\G_m\TMF_2\arrow[r,hook]& \Z_2\arrow[r]& 0\\
\pi_1\G_m\TMF_2\arrow[r]&\Z/2\{\eta\} \arrow[r]& \Z/2\{\eta\} \\
\pi_0\G_m\TMF_2\arrow[r]&(\Z/2)^\infty\arrow[r]& \pi_0\G_mL_{K(1)}L_{K(2)}\TMF. \\
\end{tikzcd}
\]
 As indicated from the notation, $\Z/2\{\eta\}$ survives in $\G_mL_{K(1)}L_{K(2)}\TMF$, so it does not lift to $\G_m\TMF_2$. 
Indeed, $\eta$ admits a strict structure in $\S_{K(1)}$, and thus its image in every $K(1)$-local $\E_\infty$-ring inherits a strict structure. All maps are the obvious ones, besides the map at the bottom right---this was shown above to coincide with the projection $(\Z/2)^\infty\oplus(\Z/2)^\infty\to (\Z/2)^\infty$ onto a summand (the source is spanned by $\{\eta, j^{2\Z+1}\eta\}$, while the target is spanned by $\{\eta, j^{-(2\mathbb{N}+1)}\eta\}$) followed by an inclusion (with cokernel $\mathrm{Ker}(\theta_0)$).  This gives
\[\G_m\TMF_2=(\Z/2)^\infty\oplus \Z_2[2]\oplus\Z_2[3].\]
The diagram for $\tmf_2$ looks identical, but this time the map on the bottom right is an isomorphism ($(\Z/2)^\infty\to (\Z/2)^\infty$, the source also being spanned by$\{\eta, j^{-(2\mathbb{N}+1)}\eta\}$) followed by an inclusion. Thus we get
\[\G_m\tmf_2= \Z_2[2]\oplus\Z_2[3].\]

\subsubsection{$p=3$}
Again, in the critical range we have
\[
\begin{tikzcd}
    \pi_3\G_m\TMF_3\arrow[r,"\sim"]&\Z_3 \arrow[r]& 0 \\
\pi_2\G_m\TMF_3\arrow[r,"\sim"]& \Z_3\arrow[r]& {0} \\
\pi_1\G_m\TMF_3\arrow[r,"\sim"]&0 \arrow[r]& { 0} \\
\pi_0\G_m\TMF_3\arrow[r,hook]&\F_3^\times\oplus\F_3^\times\arrow[r]& { \pi_0\G_mL_{K(1)}L_{K(2)}\TMF}. \\
\end{tikzcd}
\]
In terms of the right bottom corner, we find a subgroup $\F_3^\times$ corresponding to the Teichm\"{u}ller lifts, and the map $\F_3^\times\oplus \F_3^\times \to \pi_0\G_mL_{K(1)}L_{K(2)}\TMF$ factors through this subgroup, whence it is the product of two pieces of identity maps. This gives
\[\G_m\TMF_3=\F_3^\times\oplus \Z_3[2]\oplus\Z_3[3].\]
The LES is identical for $\tmf_3$, giving
\[\G_m\tmf_3=\G_m\TMF_3.\]

\subsubsection{$p\geq 5$}
Just like before, we get
\[
\begin{tikzcd}
    \pi_3\G_m\TMF_p\arrow[r,"\sim"]&\Z_p^{N_p} \arrow[r]& 0 \\
\pi_2\G_m\TMF_p\arrow[r,"\sim"]& \Z_p^{N_p}\arrow[r]& 0 \\
\pi_1\G_m\TMF_p\arrow[r,"\sim"]& 0 \arrow[r]& 0 \\
\pi_0\G_m\TMF_p\arrow[r,hook]& \F_p^\times{ \oplus}(\F_p^\times)^{\oplus(2N_p-\#S_p) }\oplus(\F_{p^2}^\times)^{\oplus(\#S_p-N_p)}\arrow[r]& {\pi_0\G_mL_{K(1)}L_{K(2)}\TMF}. \\
\end{tikzcd}
\]
Recall that $N_p$ is the number of $\mathrm{Gal}(\Fpb/\F_p)$-orbits of isomorphism classes $S_p$ of supersingular elliptic curves over $\Fpb$. For the undetermined maps among $\pi_0$, we find a subgroup $$(\F_p^\times)^{\oplus(2N_p-\#S_p) }\oplus(\F_{p^2}^\times)^{\oplus(\#S_p-N_p)}\subset \pi_0\G_mL_{K(1)}L_{K(2)}\TMF,$$
such that each factor corresponds to the Teichm\"{u}ller lifts of a connected component of the affine scheme $\mathrm{Spec}(\pi_0\G_mL_{K(1)}L_{K(2)}\TMF)$. Consequently, by the identification of $\pi_0\G_mL_{K(1)}\TMF$, the bottom right map factors through this subgroup, where the restriction to the first factor $\F_p^\times$ is the diagonal map, while the restriction to the second one is the identity map.
This gives
\[\G_m\TMF_p=\F_p^\times\oplus\Z_p^{N_p}[2]\oplus\Z_p^{N_p}[3].\]
Just as at $p=3$, we find that $\G_m\tmf_p$ is isomorphic to $\G_m\TMF_p$.

\subsection{Arithmetic fracture}\label{sec7.2}
Next, we apply $\G_m$ to the arithmetic fracture square of $\TMF$, to assemple the $p$-complete information of the previous section:
\[
\begin{tikzcd}
    \G_m\TMF\arrow[r]\arrow[d]&\arrow[d]\prod_p\G_m\TMF_p \\
\G_m(\Q\otimes\TMF)\arrow[r]&\G_m(\Q\otimes\prod_p\TMF_p).
\end{tikzcd}
\]
The $\G_m's$ of the rational spectra in the bottom row are easy: $\Q\otimes \TMF=\Q[c_4,c_6, \Delta^{-24}]$, and $\Q\otimes\TMF_p=\Q[c_4,c_6, \Delta^{-24}]_p$. We get an LES that reads

\[
\begin{tikzcd}
\pi_{4k}\G_m\TMF\arrow[r,hook]& \Q[x]\oplus 0\arrow[r]& \Q\otimes\prod_p\Z[x]_p\\
\pi_3\G_m\TMF\arrow[r]&0\oplus\prod_p\pi_3\G_m\TMF_p \arrow[r]& 0 \\
\pi_2\G_m\TMF\arrow[r,"\sim"]&0\oplus\prod_p\pi_2\G_m\TMF_p\arrow[r]& 0 \\
\pi_1\G_m\TMF\arrow[r,"\sim"]&0\oplus\prod_p\pi_1\G_m\TMF_p\arrow[r]& 0 \\
\pi_0\G_m\TMF\arrow[r]&\Q^\times\oplus\prod_p\pi_0\G_m\TMF_p\arrow[r]& \left(\Q\otimes\prod_p\Z[x]_p\right)^\times,\\
\end{tikzcd}
\]
where the top row will also be repeated in all degrees 0 modulo 4, and everything not depicted in this diagram is zero.

Filling in the results from above, this LES becomes
\[
\begin{tikzcd}
\pi_4\G_m\TMF\arrow[r,hook]& \Q[x]\oplus 0\arrow[r]& \Q\otimes\prod_p\Z[x]_p\\
\pi_3\G_m\TMF\arrow[r]&\prod_p\Z_p^{N_p} \arrow[r]& 0 \\
\pi_2\G_m\TMF\arrow[r,"\sim"]&\prod_p\Z_p^{N_p} \arrow[r]& 0 \\
\pi_1\G_m\TMF\arrow[r,"\sim"]&0\arrow[r]& 0 \\
\pi_0\G_m\TMF\arrow[r]&\Q^\times\oplus(\Z/2)^\infty\oplus\prod_p\F_p^\times\arrow[r]& \left(\Q\otimes\prod_p\Z[x]_p\right)^\times. \\
\end{tikzcd}
\]

Recall from the introduction that $A$ is the rational vector space defined as the cofiber of the obvious map
\[\Q\otimes \Z[x]\rightarrow \Q\otimes\prod_p\Z[x]_p,\]
and $N_p$ is the number of $\mathrm{Gal}(\Fpb/\F_p)$-orbits of isomorphism classes of supersingular elliptic curves over $\Fpb$. Thus we finally arrive at our main theorem.

\begin{thm}\label[thm]{mainthm}
    There exists an equivalence of connective $\Z$-modules, where the numbers on the right indicate the degree of the corresponding Eilenberg--Maclane spectrum, and $k> 1$,
\begin{align*}
    \G_m\TMF =
    &(\Z/2)^\infty &[0] \\
    &\prod_p\Z_p^{N_p} &[2] \\
    &\prod_p\Z_p^{N_p}\oplus A &[3] \\
    & A&[4k-1].
\end{align*}
Moreover, the map $\G_m\tmf\to\G_m\TMF$ exhibits the source as the connected cover of the target.
\end{thm}

\begin{rem}
    We also see that the map $\G_m\TMF\to\glone\TMF$ is zero on homotopy groups. Indeed, throughout the paper, the only elements that admitted strict structures were the Teichm\"{u}ller lifts, and $\eta$---the former do not even exist before $p$-completion, and the strict structure on the latter only exists after $K(1)$-localization. 
\end{rem}

\bibliographystyle{halpha}

\bibliography{references}

\end{document}